\documentclass[11pt,a4paper]{article}
\usepackage{amsmath}
\usepackage{amssymb}  
\usepackage{epsfig}
\usepackage{amsthm}   
\usepackage[mathcal]{eucal}
\usepackage{url}
 \usepackage{mathrsfs}

\newcommand{\inner}[2]{\langle{#1},{#2}\rangle}
\newcommand{\Inner}[2]{\left\langle{#1},{#2}\right\rangle}

\newcommand{\tos}{\rightrightarrows} 
\newcommand{\ovl}{\overline}

\newtheorem{theorem}{Theorem}[section]

\DeclareMathOperator{\graph}{gra}

\begin{document}
\title{A non-type (D) linear isometry}
\author{ O. Bueno\thanks{IMPA, Estrada Dona Castorina 110,
    22460-320 Rio de Janeiro, Brazil ({\tt obueno@impa.br}) \ Partially
    supported by CAPES.}
\and
B. F. Svaiter \thanks{IMPA, Estrada Dona Castorina 110,
    22460-320 Rio de Janeiro, Brazil ({\tt benar@impa.br}) \ Partially
    supported by CNPq grants 300755/2005-8, 475647/2006-8 and by
    PRONEX-Optimization.}}
\date{\today}
\maketitle

\begin{abstract}
Previous constructions of non-type (D) maximal monotone operators were based
on the non-type (D) operators introduced by Gossez, and 
the construction of such operators or the proof that they were non-type (D)
were not straightforward.
The aim of this paper is to present a very simple non-type (D) linear isometry.

\medskip
\noindent  \textbf{2010 Mathematics Subject Classification:} 47H05, 49J52, 47N10.

\medskip

\noindent  {\bfseries Keywords:} monotone operators, non-reflexive Banach spaces, type (D)
\end{abstract}
\pagestyle{plain}

\section{Introduction}

Maximal Monotone operators were defined and used in the early sixties
as a theoretical framework for the study of electrical networks, and, later on,
for the study of non-linear partial differential equations.
The first works on monotone operators were due to
Zarantonello~\cite{zaran1960}, Minty~\cite{Minty60}, Kato~\cite{Kato64},
Browder~\cite{Browder68}, Rockafellar~\cite{Rock69-LB}, 
Br\'ezis~\cite{Brezis73}, among others. 
Since then, monotone operators were object of intense study. See \cite{Bor10}
for a survey on the subject.

Fitzpatrick proved that any maximal monotone operator can be represented by
convex functions by providing the explicit formula for one of these functions.
Therefore, is quite natural to ask the inverse question: under which conditions
a convex function represents a maximal monotone operator.
Characterizations of the convex functions which represent such
operators in reflexive Banach spaces where presented in \cite{BS}. In
non-reflexive Banach spaces, a characterization of convex functions
which represents a sub-class of maximal monotone operators (those of
type (D)) where presented in \cite{BM1}.  In the non-reflexive space
$\ell^1$, non-type (D) maximal monotone operators where presented by
Gossez in \cite{JPGos3,JPGos2}. Since then, examples of non-type (D)
(maximal monotone) operators in $c_0$ where presented in
\cite{SvaBue11-2}, and such examples in James spaces where presented
in \cite{BBWY2011}. So far, the examples of non-type (D) maximal
monotone operators were constructed using always Gossez's original
examples and ideas in their construction.

The technicalities presented in the construction of previous
non-type (D) operator may lead one to believe that operators of
this type are inherently complex and/or pathological.
Our aim is to present very simple linear isometry which happens to
be a non-type (D) maximal monotone operator.

\section{Notation and basics definitions}
\label{sec:nbd}

Let $A$ and $B$ be arbitrary sets.
A point-to-set operator $T:A\tos B$  of $A$ in $B$, is a triple
$(A,B,G)$ 
where $G\subset A\times B$. The set $G$, called the \emph{graph} of $T$,
is denoted as
\[
\graph(T)=G
\]
and, for $a\in A$ 
\[
T(a)=\{b\;|\; (a,b)\in \graph(T)\}
\]
Hence, $T:A\tos B$ may also be regarded as a map of $A$ in to $\wp(B)$, the power set of $B$,
and be denoted as $T:A\to\wp(B)$.
From now on we identify a map
$F:A\to B$ with the operator $F:A\tos B$ with the same graph.

Let $X$ be a real Banach space, with topological dual $X^*$.
We will use the notation
\[
\inner{x}{x^*}=\inner{x^*}{x}=x^*(x),\qquad \forall x\in X,x^*\in X^*.
\]
The weak-$*$ topology of $X^*$ is the smallest topology (in $X^*$), in
which the maps $X^*\ni x^*\mapsto\inner{x}{x^*}$ are continuous for
each $x\in X$.
The bidual of $X$ is $X^{**}=(X^*)^*$ and the canonical injection of
$X$ in to $X^{**}$ is
\[
J:X\to X^{**}, \qquad \inner{x^*}{J(x)}=\inner{x^*}{x}\;\forall x^*\in X^*.
\]
Note that this map is a linear isometry.  From now on, we identify $X$
with its image under the canonical injection of $X$ into $X^{**}$. The
space $X$ is \emph{non-reflexive} if $J$ is not onto, which under the
above convention means $X\subsetneq X^{**}$.

A point-to-set operator $T:X\tos X^*$ (respectively, $T:X^*\tos X$) is
monotone if for any $(x,x^*),(y,y^*)\in \graph(T)$ (respectively, for
any $(x^*,x),(y^*,y)\in\graph(T)$)
\[
\inner{x-y}{x^*-y^*}\geq 0
\]
and is \emph{maximal monotone} if it is monotone and its graph is
maximal, with respect to the partial order of inclusion, in the family
of graphs of maximal monotone operators from $X$ to $X^*$
(respectively, from $X^*$ to $X$)

Define, for $T:X\tos X^*$ the operator $\ovl T:X^{**}\tos X^*$ which
graph is given by the limits, in the weak-$*\times$strong topology of
$X^{**}\times X^*$, of \emph{bounded} nets of elements in the graph of
$T$.

A maximal monotone operator $T:X\tos X^*$ is of type (D)~\cite{JPGos0},
if every point $(x^{**},x^*)\in X^{**}\times X^*$ such that
\[
\inner{x^*-y^*}{x^{**}-y}\geq 0,\quad\forall\,(y,y^*)\in\graph(T),
\]
is contained in the graph of $\ovl{T}$. This is equivalent to the fact of $\ovl T:X^{**}\tos X^*$ being maximal monotone. 
\section{A self-canceling non-type (D) maximal monotone linear isometry}

The next theorem is our main result.

\begin{theorem}
 \label{th:ce}
 Let $X$ be a non-reflexive real Banach space. The operator
 \[
 T:X\times X^*\to X^{*}\times X^{**}, \qquad T(x,x^*)=(-x^*,x)
 \]
 is a non-type (D) maximal monotone linear isometry with infinitely
 many maximal monotone extensions to $X^{**}\times X^{***}\tos
 X^*\times X^{**}$.
\end{theorem}
\begin{proof}
  It follows trivially from its definition that $T$ is a linear
  monotone isometry.
  Since $T$ is a continuous monotone map, it is maximal monotone.

  Take $(p,p^*)\in \graph \ovl T$. This means that there exists a
  bounded net of elements in the graph of $T$
  \[
  \big\{ \left((x_i,x^*_i),(y^*_i,y^{**}_i)\right)\big\}_{i\in I}
  \]
  which converges in the weak-$*\times$strong topology to $(p,p^*)$.
  Using the definition of $T$ we have
  \[
  x^*_i=-y^*_i,\qquad y^{**}_i=x_i.
  \]
  Since $\{(y^*_i,y^{**}_i) \}$ converges in the norm topology,
  $\{(x_i,x^*_i)\}$ also converges in the strong topology and its
  limits belong to $X\times X^*$.  Therefore for some $x\in X$,
  $x^*\in X^*$,
  \[ p=(x,x^*)\in X\times X^*,\qquad p^*=(-x^*,x)
  \]
  Using again the definition of $T$ we have $(p,p^*)\in \graph(T)$. Altogether we proved
  that
  \[
  \graph(\ovl T)=\graph(T).
  \]

  Now we will prove that $\ovl T:X^{***}\times X^{**}\tos X^*\times
  X^{**}$ has infinitely many maximal monotone extensions.  Since $X$
  is non-reflexive and a closed subspace of $X^{**}$, there exist
  $x_0^{**}\in X^{**}\setminus X$ and $w_0\in X^{***}$ such that
  \begin{equation}
    \label{eq:xw}
   \inner{x^{**}_0}{w_0}=1,
   \qquad w_0(x)=0,\;\forall x\in X
  \end{equation}
  Define, for each $t\in (0,\infty)$,
  \[
  p_t=
  (t x^{**}_0,(1/t)w_0),\qquad q_t=(0,tx^{**}_0) .
  \]
  Take $(x,x^*)\in X\times X^*$.
  Direct calculation yields:
  \begin{align*}
    \inner{(x,x^*)-p_t}{(-x^*,x)-q_t}&=\inner{(x-tx_0^{**},x^*-(1/t)w_0)}{(-x^*,x-tx_0^{**})}\\
    &=\inner{x-tx_0^{**}}{-x^*}+\inner{x^*-(1/t)w_0}{x-tx_0^{**}}=1
  \end{align*}
  where the last equality follows from \eqref{eq:xw}.
  Therefore, the operator $A_t:X^{**}\times X^{***}\tos X^{*}\times X^{**}$
  with graph 
  \[
  \graph(A_t)=\graph{T}\cup\{(p_n,q_t)\}.
  \]
  is monotone.
  Hence, for each $t>0$, there exists
  a maximal monotone extension of $A_t$, say $B_t$. If $t,s>0$, $t\neq s$, then
  \begin{align*}
  \inner{p_t-p_s}{q_t-q_s}&= \Inner{\bigg((t-s)x^{**}_0,(1/t-1/s)w_0\bigg)}
                       {\bigg(0,(t-s)x_0^{**}\bigg)}\\
                       &=(t-s)(1/t-1/s) =-\frac{(t-s)^2}{ts}<0
  \end{align*}
  which prove that $(p_t,q_t)\in B_t\setminus B_s$
  and, in particular $B_t\neq B_s$.
\end{proof}

Observe that \emph{if} $X$ is a James's space, then $X\times X^*$ is a
non-reflexive Banach spaces which is isometric to its dual. Indeed, if
$A:X\to X^{**}$ is an isometry, then 
\[
\mathbb{A}:X\times X^*\to X^*\times X^{**},\qquad \mathbb{A}(x,x^*)=
(x^*, A(x))
\]
is such an isometry.


\begin{thebibliography}{10}

\bibitem{BBWY2011}
H.~H. Bauschke, J.~M. Borwein, X.~Wang, and L.~Yao.
\newblock Construction of pathological maximally monotone operators on
  non-reflexive {B}anach spaces.
\newblock {\em ArXiv e-prints}, Aug. 2011.

\bibitem{Bor10}
J.~M. Borwein.
\newblock Fifty years of maximal monotonicity.
\newblock {\em Optim. Lett.}, 4(4):473--490, 2010.

\bibitem{Brezis73}
H.~Brezis.
\newblock {\em Op{\'e}rateurs maximaux monotones et semi-groupes de
  contractions dans les espaces de {H}ilbert}.
\newblock North-Holland Publishing Co., Amsterdam, 1973.
\newblock North-Holland Mathematics Studies, No. 5. Notas de Matem{\'a}tica
  (50).

\bibitem{Browder68}
F.~E. Browder.
\newblock Nonlinear maximal monotone operators in {B}anach space.
\newblock {\em Math. Ann.}, 175:89--113, 1968.

\bibitem{SvaBue11-2}
O.~Bueno and B.~F. Svaiter.
\newblock A non-type {(D)} operator in $c_0$, Mar. 2011.
\newblock Accepted by Math. Program.

\bibitem{BS}
R.~S. Burachik and B.~F. Svaiter.
\newblock Maximal monotone operators, convex functions and a special family of
  enlargements.
\newblock {\em Set-Valued Anal.}, 10(4):297--316, 2002.

\bibitem{JPGos0}
J.-P. Gossez.
\newblock Op{\'e}rateurs monotones non lin{\'e}aires dans les espaces de
  {B}anach non r{\'e}flexifs.
\newblock {\em J. Math. Anal. Appl.}, 34:371--395, 1971.

\bibitem{JPGos3}
J.-P. Gossez.
\newblock On a convexity property of the range of a maximal monotone operator.
\newblock {\em Proc. Amer. Math. Soc.}, 55(2):359--360, 1976.

\bibitem{JPGos2}
J.-P. Gossez.
\newblock On the extensions to the bidual of a maximal monotone operator.
\newblock {\em Proc. Amer. Math. Soc.}, 62(1):67--71 (1977), 1976.

\bibitem{Kato64}
T.~Kato.
\newblock Demicontinuity, hemicontinuity and monotonicity.
\newblock {\em Bull. Amer. Math. Soc.}, 70:548--550, 1964.

\bibitem{Minty60}
G.~J. Minty.
\newblock Monotone networks.
\newblock {\em Proc. Roy. Soc. London. Ser. A}, 257:194--212, 1960.

\bibitem{Rock69-LB}
R.~T. Rockafellar.
\newblock Local boundedness of nonlinear, monotone operators.
\newblock {\em Michigan Math. J.}, 16:397--407, 1969.

\bibitem{BM1}
B.~F. Svaiter and M.~{Marques Alves}.
\newblock {B}ronsted-{R}ockafellar property and maximality of monotone
  operators representable by convex functions in non-reflexive {B}anach spaces.
\newblock {\em J. Convex Anal.}, 15:693--706, 2008.

\bibitem{zaran1960}
E.~Zarantonello.
\newblock {\em Solving functional equations by contractive averaging}, volume
  160 of {\em MRC technical summary report}.
\newblock Mathematics Research Center, United States Army, Univ. of Wisconsin,
  1960.

\end{thebibliography}
\end{document}